\documentclass[11pt]{amsart}
\usepackage{amssymb, amscd}

\newcommand\lieh{{\mathfrak h}}

\newcommand\lieso{{\mathfrak so}}

\newcommand\CC{\mathbb C}

\newcommand\RR{\mathbb R}

\newcommand\NN{\mathbb N}
\newcommand\HH{\mathbb H}

\newcommand\GL{{\mathrm{GL}}}
\newcommand\SO{{\mathrm{SO}}}
\newcommand\OO{{\mathrm{O}}}
\newcommand\SU{{\mathrm{SU}}}
\newcommand\U{{\mathrm{U}}}
\newcommand\diag{{\mathrm{diag}}}

\newcommand\End{\operatorname{End}}

\newtheorem{thm}{Theorem}[section]

\newtheorem{prop}[thm]{Proposition}
\newtheorem{cor}[thm]{Corollary}

\theoremstyle{definition}
\newtheorem{defn}[thm]{Definition}

\theoremstyle{remark}

\title[Sphere Vs Projective Space]{Spherical Functions:\\ The Spheres {\footnotesize Vs.} The Projective Spaces}
\author{J. Tirao}
\author{I. Zurri\'an}
\address{CIEM-FaMAF, Universidad Nacional de C\'or\-do\-ba,
C\'or\-do\-ba~5000, Argentina}
\email{tirao@famaf.unc.edu.ar}
\email{zurrian@famaf.unc.edu.ar}
\thanks{This paper was partially supported by CONICET, PIP 112-200801-01533.}
\begin{document}

\begin{abstract}
In this paper we establish a close relationship between the spherical functions of the $n$-dimensional sphere $S^n\simeq\SO(n+1)/\SO(n)$ and the spherical functions of the $n$-dimensional real projective space $P^n(\mathbb{R})\simeq\SO(n+1)/\mathrm{O}(n)$. In fact, for $n$ odd a function on $\SO(n+1)$
is an irreducible spherical function of some type $\pi\in\hat\SO(n)$ if and only if it is an irreducible spherical function of some type $\gamma\in\hat {\mathrm{O}}(n)$. When $n$ is even this is also true for certain types, and in the other cases we exhibit a clear correspondence between the irreducible spherical functions of both pairs $(\SO(n+1),\SO(n))$ and $(\SO(n+1),\mathrm{O}(n))$. Summarizing, to find all spherical functions of one pair is equivalent to do so for the other pair.

\end{abstract}
\maketitle
\section{Introduction.}

The theory of spherical functions dates back to the classical papers of \'E. Cartan and H. Weyl; they showed that spherical harmonics arise in a natural way from the study of functions on the $n$-dimensional sphere $S^n=\SO(n+1)/\SO(n)$. The first general results in this direction were obtained in 1950 by Gelfand who considered zonal spherical functions of a Riemannian symmetric space $G/K$.

The general theory of scalar valued spherical functions of arbitrary type, associated to a pair $(G,K)$ with $G$ a locally compact group and $K$ a compact subgroup, goes back to Godement and Harish-Chandra. Later, in \cite{T77} the attention was focused on the underlying matrix valued spherical functions defined as solutions of an integral equation, see Definition \ref{def}. These two notions are related by the operation of taking traces.

A first thorough study of irreducible spherical functions of any $K$-type was accomplished in the seminal work of the complex projective plane $P^2(\CC)=\SU(3)/\U(2)$ in \cite{GPT02a}. This study was generalized to the complex projective space $P^n(\CC)$, see \cite{PT12b}.

Let us remember the definition of spherical function:

\noindent Let $G$ be a locally compact unimodular group and let
$K$ be a compact subgroup of $G$. Let $\hat K$ denote
the set of all
equivalence classes of complex finite dimensional
irreducible representations of $K$; for each
$\delta\in \hat K$, let
$\xi_\delta$ denote the character of $\delta$,
$d(\delta)$ the degree of $\delta$, i.e. the dimension
of any representation in
the class $\delta$, and
$\chi_\delta=d(\delta)\xi_\delta$. We shall choose once
and for all the Haar measure $dk$ on
$K$ normalized by $\int_K dk=1$.

\noindent We shall denote by $V$ a finite dimensional vector
space over the field $\CC$ of complex numbers and by
$\End(V)$ the space
of all linear transformations of $V$ into $V$.
Whenever we refer to a topology on such a vector space
we shall be talking
about the unique Hausdorff linear topology on it.

\begin{defn} \label{def} A spherical function $\Phi$ on $G$ of type
$\delta\in \hat K$ is a continuous function on $G$ with values in
$\End(V)$ such that
\begin{enumerate} \item[i)] $\Phi(e)=I$ ($I$= identity
transformation).

\item[ii)] $\Phi(x)\Phi(y)=\int_K
\chi_{\delta}(k^{-1})\Phi(xky)\, dk$, for all $x,y\in
G$.
\end{enumerate}
\end{defn}
\noindent The reader can find a number of general results in \cite{T77} and
\cite{GV88}.

It is known that the compact connected symmetric spaces of rank one are of the form $X\simeq G/K$, where $G$ and $K$ are as follows:
\begin{enumerate}
\item[i)] {\ \hbox to 3cm{$G=\mathrm{SO}(n+1),$\hfill}\ \hbox to 4cm{ $K=\mathrm{SO}(n)$,\hfill}\ \hbox to 4cm{ $X=S^n.$\hfill}}
\item[ii)] {\ \hbox to 3cm{$G=\mathrm{SO}(n+1),$\hfill}\ \hbox to 4cm{ $K=\mathrm{O}(n)$,\hfill}\ \hbox to 4cm{ $X=P^n(\RR)$.\hfill}}
\item[iii)] {\ \hbox to 3cm{$G=\mathrm{SU}(n+1),$\hfill}\ \hbox to 4cm{ $K=\mathrm{S}(\mathrm{U}(n)\times\mathrm{U}(1))$,\hfill}\ \hbox to 4cm{ $X=P^n(\CC)$.\hfill}}
 \item[iv)] {\ \hbox to 3cm{$G=\mathrm{Sp}(n+1),$\hfill}\ \hbox to 4cm{ $K=\mathrm{Sp}(n)\times\mathrm{Sp}(1)$,\hfill}\ \hbox to 4cm{ $X=P^n(\HH)$.\hfill}}
\item[v)]   {\ \hbox to 3cm{$G=F_{4(-52)},$\hfill}\ \hbox to 4cm{ $K=\mathrm{Spin}(9)$,\hfill}\ \hbox to 4cm { $X=P^{2}(Cay)$.\hfill}}
\end{enumerate}
The zonal (i.e. of trivial $K$-type) spherical functions on $X\simeq G/K$ are the eigenfunctions of the Laplace-Beltrami operator that only depend on the distance $d(x,o)$, $x \in X$, where $o$ is the origin of $X$. In each case we call them $\varphi_0, \varphi_1, \varphi_2, \ldots, $ with $\varphi_0=1$, and let $\varphi_j^*(\theta)$ be the corresponding function induced on $[0, L]$ by $\varphi_j$, where $L$ is the diameter of $X$.  These functions turn
 out to be Jacobi polynomials
$$\varphi_j^*(\theta) =c_j\, P_j^{(\alpha,\beta)}(\cos \lambda \theta),$$
with $c_j$ defined by the condition $\varphi_j(0)=1$ and $\lambda$, $\alpha$ and $ \beta $  depending on the pair $(G,K)$. By renormalization of the distance we can assume $\lambda=1$ and $L=\pi$, (cf. \cite[p. 171]{H62}). Now we quote the following list from \cite[p. 239]{K73}:
\begin{enumerate}
\item[i)]  {\ \hbox to 4cm{$ G/K\simeq  S^n:$\hfill}                 \ \hbox to 3.5cm{$\alpha=(n-2)/2$,\hfill}$\beta=(n-2)/2$.}
\item[ii)] {\ \hbox to 4cm{$ G/K\simeq  P^n(\RR):$\hfill}            \ \hbox to 3.5cm{$\alpha=(n-2)/2$,\hfill}$\beta=-1/2$.}
\item[iii)]{\ \hbox to 4cm{$ G/K\simeq  P^n(\CC):$\hfill}            \ \hbox to 3.5cm{$\alpha=n-1$,\hfill}$\beta=0$.}
\item[iv)] {\ \hbox to 4cm{$ G/K\simeq  P^n(\mathbb H):$\hfill}      \ \hbox to 3.5cm{$\alpha=2n-1$,\hfill}$\beta=1$.}
\item[v)]  {\ \hbox to 4cm{$ G/K\simeq  P^2(Cay):$\hfill}            \ \hbox to 3.5cm{$\alpha=7$,\hfill}$\beta=3$.}
\end{enumerate}


Therefore, at first sight, the zonal spherical functions on the sphere and on the real projective space seem to be two completely different families (to clarify this point see the Appendix). But we prove in this paper that to know all the spherical functions associated to the $n$-dimensional sphere is equivalent to know them for the $n$-dimensional real projective space. Precisely, we state a direct relation between matrix valued spherical functions of the pair $(\mathrm{SO}(n+1),\mathrm{SO}(n))$ and of the pair $(\mathrm{SO}(n+1),\mathrm{O}(n))$.
In first place we prove that, for $n$ even the spherical functions of the sphere and the spherical functions of the real projective space are the same, i.e., a function $\Phi$ on $\mathrm{SO}(n+1)$ is an irreducible spherical function of type $\pi\in \hat{\mathrm{SO}}(n)$ if and only if there exists $\gamma\in\hat{\mathrm{O}}(n)$ such that the function $\Phi$ is a spherical function of type $\gamma$. When $n$ is odd there are some particular cases in which one has the same situation as when $n$ is even, and we show that these cases are easily distinguished by looking at the highest weight of the corresponding $\SO(n)$-types. For the generic cases we show how every irreducible spherical function of the projective space is explicitly related with two spherical functions of the sphere, see Theorem \ref{Matrix}.

 An immediate consequence of this paper is obtained by combining it with \cite{PTZ12}, where all irreducible spherical functions of the pair $(\SO(4),$ $\SO(3))$ are studied and exhibited. Therefore, by applying Theorem \ref{par} we also know all spherical functions of the pair $(\SO(4),\mathrm{O}(3))$, which as functions on $\SO(4)$ are the same as those of the pair $(\SO(4),\mathrm{SO}(3))$.

\section{Representations.}
Spherical functions of type $\delta\in\hat K$ arise in a natural way upon
considering representations of $G$. If $g\mapsto \tau(g)$ is a
continuous representation of $G$, say on a finite dimensional vector
space $E$, then $$P_\delta=\int_K \chi_\delta(k^{-1})\tau(k)\, dk$$ is
a projection of $E$ onto $P_\delta E=E(\delta)$. The function
$\Phi:G\longrightarrow \End(E(\delta))$ defined by
$$\Phi(g)a=P_\delta \tau(g)a,\quad g\in G,\; a\in E(\delta),$$
is a spherical function of type $\delta$. In fact, if $a\in E(\delta)$ we have
\begin{align*}
\Phi(x)\Phi(y)a&= P_\delta \tau(x)P_\delta \tau(y)a=\int_K \chi_\delta(k^{-1})
P_\delta \tau(x)\tau(k)\tau(y)a\, dk\\
&=\left(\int_K\chi_\delta(k^{-1})\Phi(xky)\, dk\right) a.
\end{align*}

If the representation $g\mapsto \tau(g)$ is irreducible then the
spherical function $\Phi$ is also irreducible. Conversely, any irreducible
spherical function on a compact group $G$ arises in this way from a finite dimensional irreducible representation of $G$.

Now we recall how one obtains the irreducible finite dimensional representations of $\OO(n)$ from the irreducible finite dimensional representations of $\SO(n)$.
Let us take $a\in\mathrm{O}(n)$ depending on $n$:
\begin{align*}
 a&=\diag(1,1,\dots,-1), &\text{if $n$ is even,}\\
a&=\diag(-1,-1,\dots,-1), &\text{if $n$ is odd.}
\end{align*}
And let $\phi$ be the automorphism of $\SO(n)$ defined by
 $$\phi(k)=aka,$$
for all $k\in \SO(n)$. Notice that when $n$ is odd $\phi$ is trivial and $\OO(n)=\SO(n)\times F$, where $F=\{1,a\}$. Therefore in this case the irreducible finite dimensional representations of $\OO(n)$ are of the form
$\gamma=\pi\otimes 1$ or $\gamma=\pi\otimes \epsilon$ where $\pi$ is an irreducible finite dimensional representation of
$\SO(n)$ and $\epsilon$ is the nontrivial character of $F$. Thus we have the following theorem.

\begin{thm} If $n$ is odd $\OO(n)=\SO(n)\times F$, and $\hat\SO(n)\times \hat F$ can be identified with the unitary dual of $\OO(n)$,  under the bijection $([\pi],1)\mapsto [\pi\otimes1]$ and $([\pi],\epsilon)\mapsto [\pi\otimes\epsilon]$.
\end{thm}

If $n$ is even we have $\OO(n)=\SO(n)\rtimes F$. Let us denote by $V_\pi$ the vector space associated to $\pi\in \hat\SO(n)$, then set $V_{\pi_\phi}=V_\pi$ and let $\pi_\phi:\SO(n)\rightarrow\End(V_{\pi_\phi})$ be the irreducible representation of $\mathrm{SO}(n)$ given by
 $$\pi_\phi=\pi\circ\phi.$$
In this situation we shall consider two cases: $\pi_\phi\sim\pi$ in Subsection \ref{equiv}  and $\pi_\phi\nsim\pi$ in Subsection \ref{nequiv} .

\subsection{When $\pi_\phi$ is equivalent to $\pi$}\label{equiv}\

Take $A\in\GL(V)$ such that $\pi_\phi(k)=A\pi(k)A^{-1}$ for all $k\in \SO(n)$. Then
$$\pi(k)=\pi(a(aka)a)=\pi_\phi(aka)=A\pi(aka)A^{-1}=A\pi_\phi(k)A^{-1}=A^2\pi(k)A^{-2}.$$
Therefore, by Schur's Lemma, we have $A^2=\lambda I$. By changing $A$ by $\sqrt{\lambda^{-1}} A$ we may assume that $A^2=I$. Let $\epsilon_A$ be the representation of $F$ defined by \begin{equation*}
\epsilon_A(a)=A            .                                                                                                                                                                          \end{equation*}
 Now we define $\gamma=\pi\cdot\epsilon_A:\OO(n)\rightarrow\GL(V)$ by
\begin{equation}\label{gammaA}
 \gamma(kx)=\pi(k)\epsilon_A(x),
\end{equation}
 and it is easy to verify that  $\gamma$ is an irreducible representation of $\OO(n)$. Moreover, if $B$ is another solution of $\pi_\phi(k)=B\pi(k)B^{-1}$ for all $k\in \SO(n)$, and $B^2=I$, then $B=\pm A$. In fact, by Schur's Lemma, $B=\mu A$ and $\mu^2=1$. In the set of all such pairs $(\pi,A)$ we introduce the equivalence relation $(\pi,A)\sim(T\pi T^{-1},TAT^{-1})$, where $T$ is any bijective linear map from $V$ onto another vector space, and set $[\pi,A]$ for the equivalence class of $(\pi,A)$.

\begin{prop} \label{equivalent} When $n$ is even $\OO(n)=\SO(n)\rtimes F$. Let us assume that $\pi_\phi\sim\pi$.  If $\gamma=\pi\cdot\epsilon_A$, then $\gamma$ is an irreducible representation of $\OO(n)$. Moreover, $\gamma'=\pi\cdot\epsilon_{-A}$, is another irreducible representation of $\OO(n)$ not equivalent to $\gamma$. Besides, the set $\{[\pi,A]: \pi_\phi(k)=A\pi(k)A^{-1}, A^2=I\}$ can be included in $\hat\OO(n)$ via the map
$[\pi,A]\mapsto [\pi\cdot\epsilon_A]$.
\end{prop}

\begin{proof} The only thing that we really need to prove is that $\gamma'\nsim\gamma$. In fact, if $\gamma'=T\gamma T^{-1}$ for some $T\in\GL(V)$, then, since $\gamma'_{\vert \SO(n)}=\gamma_{\vert \SO(n)}=\pi$, Schur's Lemma implies that $T=\lambda I$. Hence $-A=\gamma'(a)=T\gamma(a)T^{-1}=\gamma(a)=A$ is a contradiction.
\end{proof}

\subsection{When $\pi_\phi$ is not equivalent to $\pi$}\label{nequiv}\

Assume that $n$ is even and that $\pi_\phi\nsim\pi$. Let us consider the $\SO(n)$-module $V_\pi\times V_\pi$ and define
\begin{equation}\label{gamma}
 \gamma(k)(v,w)=(\pi(k)v,\pi_\phi(k)w),\quad
\gamma(ka)(v,w)=(\pi(k)w,\pi_\phi(k)v),
\end{equation}
for all $k\in \SO(n)$ and $v,w\in V_\pi$. Then it is easy to verify that  $\gamma:\OO(n)\rightarrow\GL(V_\pi\times V_\pi)$ is a representation of $\OO(n)$.

\begin{prop}\label{equivalent2} Assume that $n$ is even and $\pi\in\hat\SO(n)$, then $\OO(n)=\SO(n)\rtimes F$. Moreover, if $\pi_\phi\nsim\pi$ we define
$\gamma:\SO(n)\times F\rightarrow\GL(V_\pi\times V_\pi)$ as in \eqref{gamma}, then $\gamma$ is an irreducible representation of $\OO(n)$. Also $\gamma':\SO(n)\times F\rightarrow\GL(V_\pi\times V_\pi)$ defined by $$\gamma'(k)(v,w)=(\phi(k)v,kw),\qquad \gamma'(ka)(v,w)=(\phi(k)w,kv),$$ for all $k\in \SO(n)$ and $v,w\in V_\pi$, is an irreducible representation of $\OO(n)$, but it is equivalent to $\gamma$. Besides, the set $\{\{[\pi],[\pi_\phi]\}: [\pi]\in \hat \SO(n)\}$ can be included in $\hat\OO(n)$ via the map
$\{[\pi],[\pi_\phi]\}\mapsto [\gamma]$.
\end{prop}
\begin{proof}
Let us prove that $\gamma$ is irreducible. Then $V_\pi\times V_\pi\sim V_\pi\oplus V_{\pi_\phi}$ as $\SO(n)$-modules defined by $\gamma$. If $V_\pi\times V_\pi=W_1\oplus W_2$ and correspondingly  $\gamma=\gamma_1\oplus\gamma_2$ as $\OO(n)$-modules, then we may assume that $W_1\sim V_\pi$ and $W_2\sim V_{\pi_\phi}$ as $\SO(n)$-modules. But
$\gamma_1(aka)=\gamma_1(a)\gamma_1(k)\gamma_1(a)$ implies that $\pi_\phi\sim\pi$, contradiction. Thus $\gamma$ is irreducible.

Finally it is easy to check that the transposition map $s: V_\pi\times V_\pi\rightarrow V_\pi\times V_\pi$ defined by $s(v,w)=(w,v)$ is an intertwing operator between $\gamma$ and $\gamma'$.
\end{proof}

\begin{thm} Assume that $n$ is even. We split $\hat\OO(n)$ into two disjoint sets: (a) $\{[\gamma]:\gamma_{\vert \SO(n)} \; {\text irreducible}\}$ and (b) $\{[\gamma]:\gamma_{\vert \SO(n)} \;{\text reducible}\}$.
\noindent\begin{enumerate} \item[(a)]
If $[\gamma]$ is in the first set and $\pi=\gamma_{\vert \SO(n)}$, then
$$\pi_\phi(k)=\pi(aka)=\gamma(aka)=\gamma(a)\pi(k)\gamma(a),$$
for all $k\in \SO(n)$. Therefore $\pi_\phi\sim\pi$ and $\gamma$ is equivalent to the representation $\pi\cdot\epsilon_A$ constructed from $\pi$ and $A=\gamma(a)$ in \eqref{gammaA}.
\item[(b)]
If $[\gamma]$ is in the second set, let $W$ be the representation space of $\gamma$. Let $V_\pi<W$ be an irreducible $\SO(n)$-module. Then $W=V_\pi\oplus V_{\pi_\phi}$ as $\SO(n)$-modules, and $\gamma$ is equivalent to the representation $\gamma'$ defined on $V_\pi\times V_\pi$ by \eqref{gamma}.
\end{enumerate}
\end{thm}
\begin{proof} The first assertion does not need to be proved. So, let us assume that $[\gamma]$ is in the second set. Let us first check that $\gamma(a)V_\pi$ is a $\SO(n)$-module: if $v\in V_\pi$ and $k\in \SO(n)$, then $\gamma(k)\gamma(a)v=\gamma(a)\gamma(aka)v\in\gamma(a)V_\pi$. Thus $V_\pi\cap\gamma(a)V_\pi$ and $V_\pi+\gamma(a)V_\pi$ are  $\OO(n)$-modules, because they are $\SO(n)$-modules stable under $\gamma(a)$. Therefore, by irreducibility we have $W=V_\pi\oplus\gamma(a)V_\pi$. Since $\gamma(k)\gamma(a)v=\gamma(a)\pi_{\phi}(k)v$ we obtain that $\gamma\vert_{\gamma(a)V_\pi}\sim\pi_\phi$, then we have $\gamma=\pi\oplus\pi_\phi$ and $W=V_\pi\oplus V_{\pi_\phi}$.

If $\pi_{\phi}$ and $\pi$ are equivalent, then there exits $A\in\GL(V_\pi)$ such that $\pi_{\phi}(k)=A\pi(k)A^{-1}$. Let us now consider the linear space $\{(v,\gamma(a)Av)\}\ne W$, and prove that it is an $\OO(n)$-module, which is a contradiction: $$\gamma(k)(v,\gamma(a)Av)=(kv,\gamma(k)\gamma(a)Av)=(kv,\gamma(a)\pi_{\phi}(k)Av)=(kv,\gamma(a)Akv)$$ and $\gamma(a)(v,\gamma(a)Av)=(Av,\gamma(a)AAv)$. Therefore $\pi_{\phi}\nsim\pi$. Finally, let us check that the representation $\gamma'$ given by \eqref{gamma} is equivalent to $\gamma$: Let $T:V_\pi\times V_\pi\rightarrow V_\pi\oplus\gamma(a)V_\pi$ be the linear map defined by $T(v,w)=v+\gamma(a)w$, then
$$T\gamma'(k)(v,w)=\pi(k)v+\gamma(a)\pi_\phi(k)w=\pi(k)v+\gamma(k)\gamma(a)w=\gamma(k)T(v,w),$$
$$T\gamma'(a)(v,w)=T(w,v)=w+\gamma(a)v=\gamma(a)(v+\gamma(a)w)=\gamma(a)T(v,w).$$
\end{proof}

\subsection{The highest weights of $\pi$ and $\pi_\phi$.}	
\

When $n$ is even it would be very useful to know when $\pi\in\hat \SO(n)$ is equivalent to $\pi_\phi$. In that direction we prove a very simple criterion in terms of the highest weight of $\pi$.

For a given $\ell\in\mathbb{N}$, we know from \cite{vilenkin} that the highest weight of an irreducible representation $\pi$ of $\mathrm{SO}(2\ell)$ is of the form ${\bf m}_\pi=(m_1, m_2, m_3, \dots , m_\ell)$ $\in\mathbb{Z}^\ell$, with $$m_1\ge m_2\ge m_3\ge\dots\ge m_{\ell-1} \ge |m_\ell|.$$
\begin{thm}\label{weights}
 If ${\bf m}_\pi=(m_1, m_2, m_3, \dots , m_\ell)$ is the highest weight of $\pi\in\hat{\mathrm{SO}}(2\ell)$ then ${\bf m}_{\pi_\phi}=(m_1, m_2, m_3, \dots , -m_\ell)$ is  the highest weight of $\pi_\phi$.
\end{thm}

The matrices $I_{ki},\, 1\le i<k\le 2\ell$, with $-1$ in the place $(k,i)$, $1$ in the place $(i,k)$ and everywhere else zero, form a basis of the Lie algebra $\mathfrak{so}(2\ell)$.
The linear span
$$\lieh={\langle I_{21},I_{43},\dots,I_{2\ell,2\ell-1}\rangle}_\CC$$
is a Cartan subalgebra of $\mathfrak{so}(2\ell,\CC)$.

Now consider
 $$H=i(x_1I_{21}+\dots+x_\ell I_{2\ell,2\ell-1})\in\lieh,$$
and let $\epsilon_j\in\lieh^*$ be defined by $\epsilon_j(H)=x_j$, $1\le j\le\ell$. Then for $1\le j<k\le\ell$, the following matrices are root vectors of $\lieso(2\ell,\CC)$:
\begin{equation}\label{rootvectors}
\begin{split}
X_{\epsilon_j+\epsilon_k}&=I_{2k-1,2j-1}-I_{2k,2j}-i(I_{2k-1,2j}+I_{2k,2j-1}),\\
X_{-\epsilon_j-\epsilon_k}&=I_{2k-1,2j-1}-I_{2k,2j}+i(I_{2k-1,2j}+I_{2k,2j-1}),\\
X_{\epsilon_j-\epsilon_k}&=I_{2k-1,2j-1}+I_{2k,2j}-i(I_{2k-1,2j}-I_{2k,2j-1}),\\
X_{-\epsilon_j+\epsilon_k}&=I_{2k-1,2j-1}+I_{2k,2j}+i(I_{2k-1,2j}-I_{2k,2j-1}).
\end{split}
\end{equation}
We choose the following set of positive roots
$$\Delta^+=\{\epsilon_j+\epsilon_k, \epsilon_j-\epsilon_k: 1\le j<k\le\ell\}$$
having so $${\bf m}_\pi=m_1\epsilon_1+m_2\epsilon_2+\ldots+m_\ell\epsilon_\ell.$$
\begin{proof}[Proof of Theorem \ref{weights}]
First we prove that the highest weight vector $v_\pi$ of the representation $\pi$ is also a highest weight vector of $\pi_\phi$:

For every root vector $X_{\epsilon_j\pm \epsilon_k}$ with $1\le j < k <\ell$ we have that $\operatorname{Ad}(a) X_{\epsilon_j\pm \epsilon_k}$ $=X_{\epsilon_j\pm \epsilon_k}$. And, when $k=\ell$ we have that $\operatorname{Ad}(a) X_{\epsilon_j\pm \epsilon_\ell}=X_{\epsilon_j\mp \epsilon_k}.$ Hence, if we denote by $\dot\pi$ and $\dot\pi_\phi$ the representations of the complexification of $\mathfrak{so}(2\ell)$ corresponding to $\pi$ and $\pi_\phi$, respectively, we have $\dot\pi\circ \operatorname{Ad}(a)=\dot\pi_\phi$ and thus
$$\dot\pi_\phi(X_{\epsilon_j\pm \epsilon_k}) v_\pi =  \dot\pi(\operatorname{Ad}(a) X_{\epsilon_j\pm \epsilon_k} ) v_\pi =  \dot\pi(X_{\epsilon_j\pm \epsilon_k}) v_\pi = 0,$$
for $1\le j <k<\ell$. When $k=\ell$ we have
$$\dot\pi_\phi(X_{\epsilon_j\pm \epsilon_\ell}) v_\pi =  \dot\pi(\operatorname{Ad}(a) X_{\epsilon_j\pm \epsilon_\ell} ) v_\pi =  \dot\pi(X_{\epsilon_j\mp \epsilon_\ell}) v_\pi = 0.$$
Therefore $v_\pi$ is a highest weight vector of $\pi_\phi$.

Notice that $\operatorname{Ad}(a)I_{2j,2j-1}=I_{2j,2j-1}$ for $1\le j<\ell$ and that $\operatorname{Ad}(a)I_{2\ell,2\ell-1}$ $=-I_{2\ell,2\ell-1}$, then
$$\dot\pi_\phi(iI_{2j,2j-1}) v_\pi =  \dot\pi(\operatorname{Ad}(a) iI_{2j,2j-1}) v_\pi =  \dot\pi(iI_{2j,2j-1}) v_\pi = m_j v_\pi,$$
for $1\le j <\ell$. When $k=\ell$ we have
$$\dot\pi_\phi(iI_{2\ell,2\ell-1}) v_\pi =  \dot\pi(\operatorname{Ad}(a) iI_{2\ell,2\ell-1} ) v_\pi =  -\dot\pi(iI_{2\ell,2\ell-1}) v_\pi =-m_\ell v_\pi.$$
Hence the highest weight of $\pi_\phi$ is
$${\bf m}=(m_1,m_2,\ldots,m_{\ell-1},-m_\ell).$$

\end{proof}

\begin{cor}
 An irreducible representation $\pi$ of $\mathrm{SO}(2\ell)$, $\ell\in\NN$, of highest weight ${\bf m}=(m_1,m_2,\ldots,m_\ell)$ is equivalent to $\pi_\phi$ if and only if $m_\ell=0$.
\end{cor}

\section{Spherical Functions}
Let $(V_\tau,\tau)$ be a unitary irreducible representation of $G=\SO(n+1)$ and $(V_\pi,\pi)$ a unitary irreducible representation of $\SO(n)$.

Let us assume that $n$ is odd. Then $\OO(n)=\SO(n)\times F$ and the irreducible unitary representations of $\OO(n)$ are of the form $\gamma=\pi\otimes 1$ or $\gamma=\pi\otimes\epsilon$. Suppose that $\pi$ is a sub-representation of $\tau_{\vert_{\SO(n)}}$. Let us observe that $a\in\OO(n)$ as an element of $G$ becomes $-I\in G$. Clearly $\tau(-I)=\pm I$. Take $\gamma=\pi\otimes 1$ if  $\tau(-I)=I$ and $\gamma=\pi\otimes\epsilon$ if  $\tau(-I)=-I$. Then $\gamma$ is a sub-representation of $\tau_{\vert_{\OO(n)}}$. Let $\Phi^{\tau,\pi}$ and $\Phi^{\tau,\gamma}$ be, respectively, the corresponding spherical functions of $(G,\SO(n))$ and $(G,\OO(n))$.
\begin{thm}\label{par}
Assume that $n$ is odd. If $\Phi^{\tau,\pi}(-I)=I$, take $\gamma=\pi\otimes 1$, and if $\Phi^{\tau,\pi}(-I)=-I$, take $\gamma=\pi\otimes\epsilon$. Then $\Phi^{\tau,\pi}(g)=\Phi^{\tau,\gamma}(g)$ for all $g\in G$.
\end{thm}
\begin{proof}
As  $\SO(n)$-modules $V_\tau=V_\pi\oplus V_\pi^\perp$. But since $\tau(a)=\tau(-I)=\pm I$ the decomposition $V_\tau=V_\pi\oplus V_\pi^\perp$ is also an $\OO(n)$-decomposition. Hence, the $\SO(n)$-projection $P_\pi$ onto $V_\pi$ is equal to the $\OO(n)$-projection $P_\gamma$ onto $V_\pi$. Therefore $\Phi^{\tau,\pi}(g)=P_\pi\tau(g)P_\pi=P_\gamma\tau(g)P_\gamma=\Phi^{\tau,\gamma}(g)$, completing the proof.
\end{proof}

\

Let us assume now that $n$ is even, then $\OO(n)=\SO(n)\rtimes F$.
Suppose that $\pi\in\hat \SO(n)$  and that $\pi\sim\pi_\phi$. Then $\gamma=\pi\cdot\epsilon_A$, where $A\in\GL(V_\pi)$ is such that $\pi_\phi=A\pi A^{-1}, A^2=I$, is an irreducible representation of $\OO(n)$ in $V_\pi$ as we have seen in Proposition \ref{equiv}. Now we use this result to obtain the following Theorem.

\begin{thm}\label{impar} Assume that $n$ is even. Let $a=\diag(1,1,\dots,-1)\in\OO(n)$ be identified with
$a=\diag(1,1,\dots,-1,-1)\in\SO(n+1)$.  Suppose that $\pi$ is a sub-representation of $\tau_{\vert_{\SO(n)}}$ and that $\pi\sim\pi_\phi$. Set  $A=\Phi^{\tau,\pi}(a)$, and take
$\gamma=\pi\cdot\epsilon_A$. Then $\Phi^{\tau,\pi}(g)=\Phi^{\tau,\gamma}(g)$ for all $g\in G$.
\end{thm}
\begin{proof}
The first thing we have to prove is that  $A\in\GL(V_\pi)$, $\pi_\phi=A\pi A^{-1}$ and  $A^2=I$. This last property follows directly from $a^2=e$.

For all $k\in \SO(n)$ we have
\begin{equation}\label{equivalencia}
\pi_\phi(k)=\pi(aka)=\tau(aka)_{\vert V_\pi}=\tau(a)\tau(k)\tau(a)_{\vert V_\pi}.
\end{equation}
Therefore $\tau(a)V_\pi$ is a $\SO(n)$-module equivalent to $\pi_\phi$. Since $\pi_\phi\sim\pi$, and by the multiplicity one property of the pair $(\SO(n+1),\SO(n))$, we obtain that $V_\pi=\tau(a)V_\pi$. Therefore $A=\tau(a)_{\vert V_\pi}\in\GL(V_\pi)$ and $\pi_\phi=A\pi A^{-1}$.
Hence $V_\pi$ is an $\OO(n)$ submodule of $V_\tau$  and the corresponding representation is $\gamma=\pi\cdot\epsilon_A$. This implies that the $\SO(n)$-projection $P_\pi$ onto $V_\pi$ is equal to the $\OO(n)$-projection $P_\gamma$ onto $V_\pi$. Therefore $\Phi^{\tau,\pi}(g)=P_\pi\tau(g)P_\pi=P_\gamma\tau(g)P_\gamma=\Phi^{\tau,\gamma}(g)$. Finally we observe that $\Phi^{\tau,\pi}(a)=P_\pi\tau(a)P_\pi=\tau(a)_{\vert V_\pi}$, completing the proof.
\end{proof}

\

Let us assume that $n$ is even, and take $\pi\in\hat \SO(n)$ such that $\pi\nsim\pi_\phi$. Let us consider the $\SO(n)$-module $V_\pi\times V_{\pi_{\phi}}$ and define
$\gamma(k)(v,w)=(\pi(k)v,\pi_{\phi}(k)w)$, $\gamma(ka)(v,w)=(\pi(k)w,\pi_{\phi}(k)v)$ for all $k\in \SO(n)$, $v\in V_\pi$ and $w\in V_{\pi_{\phi}}$. Then $\gamma$ is an irreducible representation of $\OO(n)$ in  $V_\pi\times V_{\pi_{\phi}}$ .

\begin{thm}\label{Matrix} Assume that $n$ is even. Let $a=\diag(1,1,\dots,-1)\in\OO(n)$ be identified with
$a=\diag(1,1,\dots,-1,-1)\in\SO(n+1)$. Suppose that $\pi$ is a sub-representation of $\tau_{\vert_{\SO(n)}}$ and that $\pi\nsim\pi_\phi$. Then $\tau(a)V_\pi\sim V_{\pi_{\phi}}$ as $\SO(n)$-modules and $V_\pi\oplus \tau(a)V_\pi$ is an irreducible $\OO(n)$-submodule of $V_\tau$ equivalent to the irreducible representation $\gamma$ in
$V_\pi\times V_{\pi_{\phi}}$ constructed above. Moreover,
$$\Phi^{\tau,\gamma}(g)=\left(\begin{matrix} \Phi^{\tau,\pi}(g) & \Phi^{\tau,\pi}(ga)\\ \Phi^{\tau,\pi_\phi}(ga) & \Phi^{\tau,\pi_\phi}(g)\end{matrix}\right)$$
for all $g\in G$.
\end{thm}
\begin{proof}
That $\tau(a)V_\pi\sim V_{\pi_\phi}$ as $\SO(n)$-modules  follows from \eqref{equivalencia}. Also, if we make the identification  $V_\pi\times V_{\pi_\phi}\sim V_\pi\oplus \tau(a)V_\pi$ via the $\SO(n)$-isomorphism
$(v,w)\mapsto v+\tau(a)w$, and using again that $\pi_\phi(k)w=\tau(a)\tau(k)\tau(a)w$ (see \eqref{equivalencia}), we have
\begin{equation*}
\begin{split}
\gamma(k)(v,w)&= (\pi(k)v,\pi_\phi(k)w)=(\pi(k)v,\tau(a)\tau(k)\tau(a)w)\\
&\sim\pi(k)v+\tau(k)\tau(a)w=\tau(k)(v+\tau(a)w)
\end{split}
\end{equation*}
 for all $k\in \SO(n)$, and
$$\gamma(a)(v,w)=(w,v)\sim(w+\tau(a)v)=\tau(a)(v+	\tau(a)w).$$
This proves that $V_\pi\oplus \tau(a)V_\pi$ as an $\OO(n)$-submodule of $V_\tau$ is equivalent to the irreducible representation $\gamma$ in $V_\pi\times V_{\pi_{\phi}}$. Therefore $P_\gamma=P_\pi\oplus P_{\pi_\phi}$.

Hence, for all $g\in G$ we have,
\begin{equation*}
\begin{split}
\Phi^{\tau,\gamma}(g)&=(P_\pi\oplus P_{\pi_\phi})\tau(g)(P_\pi\oplus P_{\pi_\phi})\\
&=\Phi^{\tau,\pi}(g)+\Phi_{21}(g)+\Phi_{12}(g)+\Phi^{\tau,\pi_\phi}(g),
\end{split}
\end{equation*}
where $\Phi_{21}(g)=P_{\pi_\phi}\tau(g)P_\pi$ and $\Phi_{12}(g)=P_{\pi}\tau(g)P_{\pi_\phi}$.
Thus in matrix form we have
$$\Phi^{\tau,\gamma}(g)=\left(\begin{matrix} \Phi^{\tau,\pi}(g) & \Phi_{12}(g)\\ \Phi_{21}(g) & \Phi^{\tau,\pi_\phi}(g)\end{matrix}\right).$$
From the identity $\Phi^{\tau,\gamma}(ga)=\Phi^{\tau,\gamma}(g)\tau(a)_{|_{V_\pi\oplus \tau(a)V_\pi}}$ we get
$$\left(\begin{matrix} \Phi^{\tau,\pi}(ga) & \Phi_{12}(ga)\\ \Phi_{21}(ga) & \Phi^{\tau,\pi_\phi}(ga)\end{matrix}\right)=\left(\begin{matrix} \Phi^{\tau,\pi}(g) & \Phi_{12}(g)\\ \Phi_{21}(g) & \Phi^{\tau,\pi_\phi}(g)\end{matrix}\right)\left(\begin{matrix} 0 & I\\ I & 0\end{matrix}\right),$$
which is equivalent to $\Phi_{12}(g)=\Phi^{\tau,\pi}(ga)$ and  $\Phi_{21}(g)=\Phi^{\tau,\pi_\phi}(ga)$.
The theorem is proved.
\end{proof}

\section{Appendix}\label{appendix}

The irreducible spherical functions of trivial $K$-type of  $(\SO(n+1),\SO(n))$ and of $(\SO(n+1),\OO(n))$ are, respectively, the zonal spherical functions of $S^n$ and $P^n(\RR)$. According to our Theorems  \ref{par} and \ref{impar} the zonal spherical functions $\phi$ of  $P^n(\RR)$, as functions on $\SO(n+1)$,  coincide with those zonal spherical functions $\varphi$ of $S^n$ such that $\varphi(-I)=1$.

As we said in the Introduction  the zonal spherical functions on the $n$-dimensional sphere and on the corresponding real projective space are, respectively,  given by
\begin{align*}
\varphi_j^*(\theta) =c_j\, P_j^{\left(\tfrac{n-2}{2},\tfrac{n-2}{2}\right)}(\cos \theta), \qquad
\phi_j^*(\theta) =c'_j\, P_j^{\left(\tfrac{n-2}{2},-\tfrac{1}{2}\right)}(\cos \theta),
\end{align*}
with $c_j$, $c'_j$ scalars such that $\varphi_j(0)=1=\phi_j(0)$, and $0\le\theta\le\pi$.

In this appendix we explain this apparent inconsistency.  To begin with, we notice that in both spaces the metric is chosen normalized by the diameter  $L=\pi$. For a given $g\in \SO(n+1)$ we denote by $\theta(g)$ the distance in the sphere between $g\cdot o$ and the origin $o$, and analogously we denote by $\theta'(g)$ the distance in the projective space between $g\cdot o$ and the origin $o$. Thus, it is not difficult to see that
\begin{align*}
2\theta(g)&=\theta'(g),\hfill &\text{for } 0\le\theta(g)\le\pi/2,\\
2\pi-2\theta(g)&=\theta'(g),\hfill &\text{for } \pi/2\le\theta(g)\le\pi.
\end{align*}
Therefore we have that
\begin{equation}\label{cos}
 \cos(2\theta(g))=\cos(\theta'(g)),
\end{equation}
for any $g\in \SO(n+1)$.
In the other hand from \cite[(3.1.1)]{AAR} we know that Jacobi polynomials satisfy
\begin{equation}\label{a}
P_{2k}^{(\alpha,\alpha)}(x)=\frac{k!(\alpha+1)_{2k}}{(2k)!(\alpha+1)_{k}}  P_{k}^{(\alpha,-1/2)}(2x^2-1).
\end{equation}
Then, if we put $x=\cos(\theta(g))$ in \eqref{a} we obtain
$$
P_{2k}^{(\alpha,\alpha)}(\cos(\theta(g)))=\frac{k!(\alpha+1)_{2k}}{(2k)!(\alpha+1)_{k}}  P_{k}^{(\alpha,-1/2)}(\cos(2\theta(g))),
$$
hence, using \eqref{cos} we have that $\varphi_{2j}^*(\theta(g))=\phi_j^*(\theta'(g))$ for all $g\in\SO(n+1)$. In other words the following identity between zonal spherical functions holds: $\varphi_{2j}=\phi_j$ as functions on $\SO(n+1)$ for all $j\ge0$.

\bibliography{esfericas}{}
\bibliographystyle{amsalpha}

\end{document}